\documentclass[10pt]{amsart}
\usepackage{amssymb,amsbsy,amsmath,amsfonts,times, amscd}
\usepackage{latexsym,euscript,exscale}

\usepackage{helvet}

\title{Rigidity of commuting affine actions on reflexive Banach spaces}
\author {Christian Rosendal}
\address{Department of Mathematics, Statistics, and Computer Science (M/C 249)\\
University of Illinois at Chicago\\
851 S. Morgan St.\\
Chicago, IL 60607-7045\\
USA}
\email{rosendal.math@gmail.com}
\urladdr{http://homepages.math.uic.edu/$~$rosendal}

\date {}
\newcommand{\norm}[1]{\lVert#1\rVert}
\newcommand{\Norm}[1]{\big\lVert#1\big\rVert}

\newcommand {\Z}{\mathbb Z}

\newcommand{\eps}{\epsilon}

\newcommand{\tom} {\emptyset}

\newcommand{\til}{\rightarrow}

\newcommand {\ku} {\mathcal}

\newtheorem{thm}{Theorem}
\newtheorem{cor}[thm]{Corollary}

\newtheorem{prop} [thm] {Proposition}

\theoremstyle{definition}

\usepackage{fouriernc}

\begin{document}

\thanks{The author's research was partially supported by NSF grants DMS 0901405 and DMS 1201295}

\maketitle

\begin{abstract}
We give a simple argument to show that if $\alpha$ is an affine isometric action of a product $G\times H$ of topological groups on a reflexive Banach space $X$ with linear part $\pi$, then either $\pi(H)$  fixes a unit vector or $\alpha|_G$ almost fixes a point on $X$.

It follows that any affine isometric action of an abelian group on a reflexive Banach space $X$, whose linear part fixes no unit vectors, almost fixes points on $X$. 
\end{abstract}

\section{Linear representations and cocycles}
When $X$ is a vector space, the group of bijective affine transformations of $X$, ${\rm Aff}(X)$, can be decomposed as a semidirect product
$$
{\rm Aff}(X)=GL(X)\ltimes X,
$$
with respect to the natural action of $GL(X)$ on $X$. The product in $GL(X)\ltimes X$ is then simply $(T,x)\cdot (S,y)=(TS, Ty+x)$, while the corresponding action  of $(T,x)\in GL(X)\ltimes X$ on $X$ is  given by $(T,x)\cdot y=Ty+x$. 

Thus, an action $\alpha$ of a group $G$ by affine transformations of the vector space $X$ can be viewed as a homomorphism of $G$ into ${\rm Aff}(X)$, which thus can be split into a linear representation $\pi\colon G\til GL(X)$, called the {\em linear part of $\alpha$}, and an associated {\em cocycle} $b\colon G\til X$ such that the following {\em cocycle identity} holds,
$$
b(gf)=\pi(g)b(f)+b(g),
$$
for all $ g,f\in G$.

If, moreover, $X$ is a reflexive Banach space and $\pi\colon G\til GL(X)$ is a fixed isometric linear representation of a topological group $G$ on $X$ that is {\em strongly continuous}, i.e., such that for every $x\in X$ the map $g\in G\mapsto gx\in X$ is continuous, we can consider the corresponding vector space $Z^1(G,\pi)$ of continuous cocycles $b\colon G\til X$ associated to $\pi$. 
The subspace $B^1(G,\pi)\subseteq Z^1(G,\pi)$ consisting of those cocycles $b$ for which the corresponding affine action $\alpha$ fixes a point on $X$, i.e., for which there is some $x\in X$ such that $b(g)=x-\pi(g)x$ for all $g\in G$, is called the set of  {\em coboundaries}. Note that if $b$ is a coboundary, then $b(G)$ is a bounded subset of $X$. Conversely, if $b(G)$ is a bounded set, then any orbit $\ku O$ of the corresponding affine action is bounded and so, by reflexivity of $X$,  its closed convex hull $C=\overline{\rm conv}(\ku O)$ is a weakly compact convex set on which $G$ acts by affine isometries. It follows by the Ryll-Nardzewski fixed point theorem \cite{Ryll} that $G$ fixes a point on $C$, meaning that $b$ must be a coboundary.

Every compact set $K\subseteq G$ determines a seminorm $\norm\cdot_K$ on $Z^1(G,\pi)$ by $\norm{ b}_K=\sup_{g\in K}\norm{b(g)}$ and the family of seminorms thus obtained endows $Z^1(G,\pi)$ with a locally convex topology. With this topology, one sees that a cocycle $b$ belongs to the closure $\overline{B^1(G,\pi)}$ if and only if the corresponding affine action $\alpha=(\pi,b)$ {\em almost has fixed points}, that is, if for any compact set $K\subseteq G$ and $\eps>0$ there is some $x=x_{K,\eps}\in X$ verifying
$$
\sup_{g\in K}\norm{\big(\pi(g)x+b(g)\big)-x}=\sup_{g\in K}\norm{b(g)-\big(x-\pi(g)x\big)}<\eps.
$$

If, for any $K$, we can choose $x=x_{K,1}$ above to have arbitrarily large norm, we see that the supremum
$$
\sup_{g\in K}\Norm{\pi(g)\frac x{\norm x}-\frac x{\norm x}}<\frac{\sup_{g\in K}\norm{b(g)}+1}{\norm x}
$$
can be made arbitrarily small, which means that the linear action $\pi$ {\em almost has invariant unit vectors}.
If, on the other hand, for some $K$ the choice of $x_{K,1}$ is bounded (but non-empty), then the same bound holds for any compact $K'\supseteq K$, whereby  we find that $b(G)\subseteq X$ is a bounded set, i.e., that $b\in B^1(G,\pi)$. Thus, this shows that if $\pi$ does not almost have invariant unit vectors, the set $B^1(G,\pi)$ will be closed in $Z^1(G,\pi)$. In fact, if $b\in Z^1(G,\pi)\setminus B^1(G,\pi)$ and $\pi$ does not almost have invariant unit vectors, then for any constant $c$ there is a compact set $K\subseteq G$ such that no vector is $(\alpha(K),c)$-invariant, where $\alpha=(\pi,b)$.

Conversely, a result of A. Guichardet \cite{Gu}, valid for locally compact $\sigma$-compact $G$, states that if $\pi$ does not have invariant unit vectors and $B^1(G,\pi)$ is closed in $Z^1(G,\pi)$, then $\pi$ does not almost have invariant unit vectors.

We define the {\em first cohomology group} of $G$ with coefficients in $\pi$ to be the quotient space $H^1(G,\pi)=Z^1(G,\pi)/B^1(G,\pi)$, while the {\em reduced cohomology group} is $\overline{H^1}(G,\pi)=Z^1(G,\pi)/\overline{B^1(G,\pi)}$.

\section{Affine actions of product groups on reflexive spaces}
In the following, let $X$ be a reflexive Banach space, $G$ and $H$ be topological groups and $\pi$ be a strongly continuous linear isometric representation of $G\times H$ on $X$. We also fix a cocycle $b\in Z^1(G\times H,\pi)$ and let $\alpha$ be the corresponding affine isometric action of $G\times H$ on $X$.

\begin{prop}\label{refl}
One of the following must hold,
\begin{enumerate}
\item there is a  $\pi(H)$-invariant  unit vector,
\item for any closed convex $\alpha(H)$-invariant sets $C\subseteq X$, $\alpha|_G$ almost has fixed points on $C$. 
\end{enumerate}
\end{prop}

\begin{proof}
Assume that there are no $\pi(H)$-invariant unit vectors in $X$. Then, if $\pi^n\colon H\til GL(X^n)$ denotes the diagonal representation on $X^n=(X\oplus\ldots\oplus X)_2$, $\pi^n(H)$ has no invariant unit vectors on $X^n$. By reflexivity, for any $x\in X^n$, $C=\overline{\rm conv}(\pi^n(H)x)$ is a $\pi^n(H)$-invariant weakly compact convex subset of $X^n$ and thus, by the Ryll-Nardzewski fixed point theorem, $\pi^n(H)$ fixes a point on $C$, whereby
$0\in \overline{\rm conv}(\pi^n(H)x)$.
Therefore, for any $\eps> 0$ and $(y_1,\ldots, y_n)\in X^n$ there are $h_i\in H$ and $\lambda_i>0$, $\sum_i \lambda_i=1$, such that 
for all $k=1,\ldots ,n$,
$$
\Norm{\sum_i\lambda_i\pi(h_i)y_k}<\eps.
$$

In particular, if $C\subseteq X$ is a closed convex $\alpha(H)$-invariant set, $\eps>0$  and $K\subseteq G$ compact, fix $y\in C$ and  find $g_1,\ldots,g_n\in K$ such that $\{\alpha(g_1)y,\ldots,\alpha(g_n)y\}$
is $\frac \eps 2$-dense in $\alpha(K)y$. Choose now $h_i$ and $\lambda_i$ as above such that
$$
\Norm{\sum_i\lambda_i\pi(h_i)(y-\alpha(g_k)y)}<\frac\eps2
$$
for all $k=1,\ldots,n$. Thus, if $g\in K$, pick $k$ such that $\norm{\alpha(g)y-\alpha(g_k)y}<\frac\eps2$. 
Then, since $\Norm{\sum_i\lambda_i\pi(h_i)}\leqslant 1$, 
\[\begin{split}
\Norm{\big(\sum_i\lambda_i\alpha(h_i)y\big)-\alpha(g)\big(\sum_i\lambda_i\alpha(h_i)y\big)}
&= \Norm{\big(\sum_i\lambda_i\alpha(h_i)y\big)-\big(\sum_i\lambda_i\alpha(g)\alpha(h_i)y\big)} \\
&= \Norm{\big(\sum_i\lambda_i\alpha(h_i)y\big)-\big(\sum_i\lambda_i\alpha(h_i)\alpha(g)y\big)}\\
&= \Norm{\sum_i\lambda_i\big(\alpha(h_i)y-\alpha(h_i)\alpha(g)y\big)}\\
&= \Norm{\sum_i\lambda_i\big(\pi(h_i)y-\pi(h_i)\alpha(g)y\big)}\\
&<\Norm{\sum_i\lambda_i\pi(h_i)(y-\alpha(g)y)}\\
&<\Norm{\sum_i\lambda_i\pi(h_i)(y-\alpha(g_k)y)}+\frac\eps2\\
&<\eps.
\end{split}\]
In other words, the point $\sum_i\lambda_i\alpha(h_i)y\in C$ is $(\alpha(K),\eps)$-invariant.
\end{proof}

\begin{cor}\label{abelian}
Let $\pi$ be a strongly continuous isometric linear representation of an abelian topological group $G$ on a reflexive Banach space $X$ and suppose that $\pi(G)$ has no fixed unit vectors. Then $\overline {H^1}(G,\pi)=0$, i.e., any affine isometric action with linear part $\pi$ almost has fixed points on $X$.
\end{cor}

\begin{proof}
It suffices to consider the linear representation of $G\times G$ given by $\pi$ separately on the first and second factor and then apply Proposition \ref{refl}.
\end{proof}

Let us also note that Corollary \ref{abelian} fails for more general Banach spaces, e.g., for $\ell_1$. To see this, let $\pi$ denote the left regular representation of $\Z$ on $\ell_1(\Z)$ and let $b\in Z^1(\Z,\pi)$ be given by 
$b(n)=e_0+e_1+\ldots+e_{n-1}$. Then $\pi$ has no invariant unit vectors. Also, if $x=\sum_{n=-k}^ka_ne_n$ is any finitely supported vector, we have
\[\begin{split}
\norm{x-\alpha(1)x}
=&|a_{-k}|+|a_{-k+1}-a_{-k}|+\ldots+|a_{-1}-a_{-2}|+|a_0-a_{-1}+1|\\
&+|a_1-a_0|+\ldots+|a_k-a_{k-1}|+|a_k|\\
\geqslant&1.
\end{split}\]
So $\norm{x-\alpha(1)x}\geqslant 1$ for all $x\in \ell_1(\Z)$ and $b\notin \overline{B^1(G,\pi)}$. 

\begin{cor}\label{refl cor}
If $\alpha(G\times H)$ has no fixed point on $X$ and $\pi(G)$ and $\pi(H)$ no invariant unit vectors, then 
\begin{enumerate}
\item $\alpha|_G$ and $\alpha|_H$ almost have fixed points, and
\item $\pi|_G$ and $\pi|_H$ almost have invariant unit vectors.
\end{enumerate}
\end{cor}

\begin{proof}
Item (1) follows directly from Proposition \ref{refl}, which means that $b|_G\in \overline {B^1(G, \pi|_G)}$ and $b|_H\in \overline {B^1(H, \pi|_H)}$. However, neither $\alpha(G)$ nor $\alpha(H)$ have fixed points, i.e., $b|_G\notin {B^1(G, \pi|_G)}$ and $b|_H\notin  {B^1(H, \pi|_H)}$. For if, e.g., $\alpha(H)$ fixed a point $x\in X$, then $C=\{x\}$ would be a closed convex $\alpha(H)$-invariant set on which $\alpha|_G$ would have almost fixed points, i.e., $x$ would be fixed by $\alpha(G)$ and so $x$ would be a fixed point for $\alpha(G\times H)$, contradicting our assumptions. Thus, neither $ {B^1(G, \pi|_G)}$ nor $ {B^1(H, \pi|_H)}$ is closed, whereby (2) follows.
\end{proof}

\begin{cor}\label{decompo}
Suppose $G=G_1\times \ldots\times G_n$ is a product of topological groups and $\pi\colon G\til GL(X)$ is a linear isometric representation on a separable reflexive space $X$. Then $X$ admits a decomposition into $\pi(G)$-invariant linear subspaces
$
X=V\oplus Y_1\oplus\ldots\oplus Y_n\oplus  W
$,
such that 
\begin{enumerate}
\item $V$ is the space of $\pi(G)$-invariant vectors,
\item any $b\in Z^1(G,\pi^{Y_i})$ factors through a cocycle defined on $G_i$,
\item $Z^1(G,\pi^{W})\subseteq \overline{B^1(G_1,\pi^W)}\oplus \ldots \oplus  \overline{B^1(G_n,\pi^W)}$,
\end{enumerate}
where $\pi^W$ denotes the restriction of $\pi$ to the invariant subspace $W$ and similarly for $Y_i$.
\end{cor}

\begin{proof}
By Theorem 4.10 of \cite{FR}, for any group of linear isometries of a separable reflexive space $Y$ there is an invariant decomposition of $Y$ into the subspace of fixed points and a canonical complement. Thus, by recursion on the size of $s\subseteq \{1,\ldots,n\}$, we obtain a $\pi(G)$-invariant decomposition
$$
X=\sum_{s\subseteq \{1,\ldots,n\}}X_s,
$$
where every non-zero $x\in X_s$ is fixed by $\pi\big(\prod_{i\notin s}G_i\big)$ and by none of $\pi(G_i)$ for $i\in s$. 
So if $b\in Z^1(G,\pi^{X_s})$ and $g\in \prod_{i\notin s}G_i$, then for any $h\in \prod_{i\in s}G_i$, 
$$
b(h)+b(g)=\pi(g)b(h)+b(g)=b(gh)=b(hg)=\pi(h)b(g)+b(h),
$$
i.e., $\pi(h)b(g)=b(g)$, which implies that $b(g)=0$. It follows that if $s\neq\tom$, then any $b\in Z^1(G,\pi^{X_s})$ factors through a cocycle defined on $\prod_{i\in s}G_i$. 

Also, if $|s|\geqslant 2$, then by Corollary \ref{refl cor} we see that any $b\in Z^1(G,\pi^{X_s})$ can be written as $b=b_1\oplus\ldots\oplus b_n$, where $b_i\in \overline{B^1(G_i,\pi^{X_s})}$. Thus, if we set $V=X_\tom$, $Y_i=X_{\{i\}}$ and $W=\sum_{|s|\geqslant 2}X_s$, the result follows.
\end{proof}

Proposition \ref{refl} was shown by Y. Shalom \cite{Sh} in the special case of locally compact $\sigma$-compact $G$ and $H$ and $X=\ku H$ a Hilbert space, but by different methods essentially relying on the local compactness of $G$ and $H$ and the euclidean structure of $X$. This also provided the central lemma for the rigidity results of \cite{Sh} via the following theorem, whose proof we include for completeness.

\begin{thm}[Shalom  \cite{Sh} for locally compact $G$ and $H$]
Let $\pi\colon G\times H\til GL(\ku H)$ be a strongly continuous isometric linear representation of a product of topological groups on a Hilbert space $\ku H$ and assume that neither $\pi(G)$ nor $\pi(H)$ have invariant unit vectors. 
Then $Z^1(G\times H,\pi)=\overline{B^1(G\times H,\pi)}$ and so $\overline{H^1}(G\times H,\pi)=0$.
\end{thm}

\begin{proof}
Let $b\in Z^1(G\times H,\pi)$ be given with corresponding affine isometric action $\alpha$ and fix compact subsets $K\subseteq G$, $L\subseteq H$ and an $\eps>0$.
Then, by Proposition \ref{refl}, the closed convex $\alpha(G)$-invariant set $C\subseteq \ku H$ of $(\alpha(L),\eps/2)$-invariant points is non-empty. Similarly, there is an  $(\alpha(K),\eps/2)$-invariant point in $\ku H$.

Now, by the euclidean structure of $\ku H$,  for any $y\in \ku H$, there is a unique point $P(y)\in C$ closest to $y$ and, as $\alpha(G)$ acts by isometries on $\ku H$ leaving $C$ invariant, the map $P$ is $\alpha(G)$-equivariant, i.e., $P(\alpha(g)y)=\alpha(g)P(y)$. Moreover, using the euclidean structure again, $P$ is $1$-Lipschitz, whereby 
$$
\norm{P(y)-\alpha(g)P(y)}=\norm{P(y)-P(\alpha(g)y)}\leqslant \norm {y-\alpha(g)y},
$$
for all $y\in \ku H$ and $g\in G$. In particular, if $y\in \ku H$ is $(\alpha(K),\eps/2)$-invariant, then $P(y)$ is both $(\alpha(K),\eps/2)$ and $(\alpha(L),\eps/2)$-invariant, i.e., $P(y)$ is $(\alpha(K\times L),\eps)$-invariant. Since $K$, $L$ and $\eps$ are arbitrary, we have that $b\in \overline {B^1(G\times H,\pi)}$.
\end{proof}

U. Bader, T. Gelander, A. Furman and N. Monod \cite{furman} studied the structure of affine actions of product groups on uniformly convex spaces (a subclass of the reflexive spaces) and in this setting obtained a slightly weaker result than Shalom. Namely, if  $\pi\colon G\times H\til GL(X)$ is a strongly continuous isometric linear representation of a product of topological groups on a uniformly convex space  $X$ such that neither $\pi(G)$ nor $\pi(H)$ have invariant unit vectors, then either
\begin{enumerate}
\item[(a)] $\pi$ almost has invariant unit vectors, or
\item[(b)] $Z^1(G\times H,\pi)={B^1(G\times H,\pi)}$.
\end{enumerate}
Proposition \ref{refl} is somewhat independent of their statement and shows that one can add that $\alpha|_G$ and $\alpha|_H$ almost have fixed points to (a) above.

\end{document}